\documentclass[11pt,a4,fleqn]{article}
\usepackage{graphicx}
\usepackage{amsmath,amssymb,latexsym,graphics,epsfig}
\usepackage{hyperref}
\usepackage{color}
\usepackage{amsthm}

\newtheorem{theorem}{\bf Theorem}[section]

\newtheorem{lemma}[theorem]{\bf Lemma}
\newtheorem{corollary}[theorem]{\bf Corollary}

\newtheorem{claim}[theorem]{\bf Claim }

\newtheorem{definitiona}[theorem]{\bf Definition}
\newtheorem{question}{\bf Question}

\newenvironment{definition}{\begin{definitiona}
\rm 
}{\end{definitiona}}

\numberwithin{equation}{section}

\begin{document}
\title{{\Large Multicolor Size-Ramsey Number of Cycles}}

\author{R. Javadi$^{\textrm{a,b,1}}$, M. Miralaei$^{\textrm{b},2}$, \\[2pt]
{\small $^{\textrm{a}}$Department of Mathematical Sciences, Isfahan University of Technology},\\
{\small Isfahan, 84156-83111, Iran}\\
{\small $^{\textrm{b}}$School of Mathematics, Institute for Research in Fundamental Sciences (IPM),}\\
{\small P.O. Box 19395-5746, Tehran, Iran }\\[2pt]
{rjavadi@iut.ac.ir, m.miralaei@ipm.ir}}

\date{}
\maketitle
 \footnotetext[1] {This research is partially
carried out in the IPM-Isfahan Branch and in part supported
by a grant from IPM (No. 1400050420).} \vspace*{-0.5cm}

 \footnotetext[2] {This research was supported by a grant from IPM.} \vspace*{-0.5cm}

\begin{abstract}
Given a positive integer $ r $, the $ r $-color size-Ramsey number of a graph $ H $, denoted by $ \hat{R}(H, r) $, is the
smallest integer $ m $ for which there exists a graph $ G $ with $ m $ edges such that, in any  edge coloring of $ G $ with $ r $ colors, $G$ contains a monochromatic copy of $ H $.
Haxell, Kohayakawa and  \L uczak showed that the size-Ramsey number of a cycle
$ C_n $ is linear in $ n $ i.e. $ \hat{R}(C_n, r) \leq c_rn $, for some constant $ c_r $. Their proof, however, is based on the Szemer\'edi's regularity lemma  and so no specific constant $ c_r $ is known.
Javadi, Khoeini, Omidi and Pokrovskiy gave an alternative proof for this result which avoids using of the
regularity lemma. Indeed, they proved that if $ n $ is even, then $ c_r $ is exponential in $ r $ and if $ n $ is odd, then $ c_r $ is doubly exponential in $ r $.

\noindent
In this paper, we improve the bound $c_r$ and prove that $c_r$ is polynomial in $r$ when $n$ is even and is exponential in $r$ when $n$ is odd. We also prove that in the latter case, it cannot be improved to a polynomial bound in $r$.
More precisely, we prove that there are some positive constants $c_1,c_2$ such that for every even integer $n$, we have $c_1r^2n\leq \hat{R}(C_n,r)\leq  c_2r^{120}(\log^2 r)n$ and for every odd integer $n$, we have  
 $c_1 2^{r}n \leq \hat{R}(C_n, r)\leq c_22^{16 r^2+2\log r}n $.\\


\end{abstract}

%%%%%%%%%%%%%%%%%%%%%%%%%%%%%%%%%%%%%%%%%%%%%%%%%%%%%%%%%%%%%%%%%%%%%%%%
\section{Introduction}
For given graphs $ H_1, \ldots, H_{r} $ and a graph $G$, we say that $G$ is Ramsey for $(H_1,\ldots,H_r)$ and we write $ G \longrightarrow (H_1, \dots, H_{r}) $, if for every $ r $-edge coloring of $ G $, with colors $ 1, \dots, r $, the graph $G$ contains a monochromatic copy of $ H_i $ whose all edges are of color $i$.
In other words, for any mapping $ \varphi : E(G)\longrightarrow \{1, \dots, r\} $,  there is some integer $i$, $ 1\leq i\leq r $, and a copy $ H_i' $ of $ H_i $ in $ G $, such that $ \varphi(e)=i $, for each $e\in E(H_i') $.
A natural question is how few vertices can a graph $ G $ have, such that $ G\longrightarrow \big(H_1, \dots, H_{r}\big) $?
Frank P. Ramsey in his seminal paper \cite{Ramsey} studied this question and proved that for given graphs $ H_1, \dots, H_r $, there exists a positive integer $n$ such that for the complete graph $ K_n $ we have 
$ K_n\longrightarrow(H_1, \dots, H_r) $. The smallest such $ n $ is known as \textit{Ramsey number} of $ H_1, \dots, H_r $ and is denoted by $ R(H_1, \dots, H_r) $.
Therefore we can write
\[
R(H_1, \dots, H_r) =\min \{|V (G)| : G \longrightarrow (H_1, \dots, H_r)\}.
\]
Note that, if the minimum is achieved by a graph $ G $, then it is also achieved by a complete graph with $ |V (G)| $ vertices.
Instead of minimizing the number of vertices, one can ask for the minimum number of  edges of such a graph, i.e. can we find a graph which possibly has more vertices than $R(H_1,\ldots,H_r)$, but
has fewer edges and still is Ramsey for $ (H_1, \dots, H_r) $? How many edges suffice to construct a graph which is Ramsey for $ (H_1, \dots, H_r) $? The attempts for answering the last question gives rise to the notion of Size-Ramsey number of graphs. 
Define the \textit{$($multicolor$)$ size-Ramsey number}  
$ \hat{R}(H_1, \dots, H_r) $ to be the minimum number of edges in a graph $ G $, such that $ G $ is Ramsey for $ (H_1, \dots, H_r) $. More formally,
\[
\hat{R}(H_1, \dots, H_r)= \min\{|E(G)| : G\longrightarrow(H_1, \dots, H_r) \}.
\]
In the diagonal case, where $ H_1=\dots = H_r=H  $,
we may write $ {R}(H, r) $ for $ {R}(H_1, \dots, H_r) $, $ \hat{R}(H, r) $ for $ \hat{R}(H_1, \dots, H_r) $ and $G\longrightarrow (H)_r$ for $G\longrightarrow (H_1,\ldots, H_r)$. Moreover, for $ r=2 $ we simply write $\hat{R}(H)$ for $\hat{R}(H,2)$.\\
Since the complete graph on $ R(H_1, \dots, H_r) $ vertices is Ramsey for $(H_1,\ldots, H_r)$, we evidently have 
\begin{align}\label{Trivial}
\hat{R}(H_1, \dots, H_r)\leq \binom{R(H_1, \dots, H_r)}{2}.
\end{align}

The study of size-Ramsey numbers was initiated by Erd\H{o}s, Faudree, Rousseau and Schelp \cite{S.R.Erd} in $ 1978 $.
They introduced the notion of o-sequences as follows:
\begin{definition}
Let $ (H_n)_{n=1}^{\infty} $ be a sequence of graphs. We say that this sequence is o-sequence, whenever
\[
\lim\limits_{n\to \infty}\hat{R}(H_n, r)\binom{R(H_n, r)}{2}^{-1}=0,
\]
For such graphs $H_n$, the trivial upper bound in (\ref{Trivial}) can be substantially improved.
\end{definition}
Let $ P_n $ be the path on $ n $ vertices. The question whether $ (P_n)_{n=0}^{\infty} $
is an o-sequence was put forward in \cite{S.R.Erd}, and in \cite{Erd 1}, where Erd\H{o}s stated the following version
of the problem.
\begin{question}\label{Path}
Prove or disprove that 
\[
\dfrac{\hat{R}(P_n)}{n}\to \infty~~ {\text{and}}~~ \dfrac{\hat{R}(P_n)}{n^2}\to 0?
\]
\end{question}
Beck \cite{Beck 1} answered this question and, using probabilistic method, proved the surprising fact that $\hat{R}(P_n)< 900 n$, i.e. the size-Ramsey number of $P_n$ is linear in $n$. He also raised the question whether the size-Ramsey
number grows linearly in the number of vertices for graphs of bounded degree. This question inspired several pieces of research later. 
The linearity of the size-Ramsey number of bounded degree trees was established by Friedman and Pippenger \cite{Friedman} (see also \cite{Haxell-tree, Xin}). Haxell and Kohayakawa \cite{Haxell-cycle} and separately, Javadi, Khoeini, Omidi and Pokrovskiy \cite{JKhOP} proved the linearity of size-Ramsey number of cycles.  
An affirmative answer to Beck's question was given for other classes of graphs including powers of paths and cycles \cite{Power path} (see \cite{jie} for the multi-color case), powers of bounded degree trees \cite{Power tree} and graphs with bounded maximum degree and bounded treewidth \cite{treewidth}. 
 However, Beck's question was settled in the negative by  R\"{o}dl and Szemer\'{e}di \cite{Rodl},
who provided counterexamples of graphs of order $ n $, maximum degree $ 3 $, and size-Ramsey
number $ \Omega \big(n(\log n)^{1/60}\big) $.\\
In the last decades many successive attempts were done in order to improve the bounds on the size-Ramsey number of paths, with two colors. See
 e.g., \cite{Beck 1, Bela.B, Dudek.1} for lower bounds, and \cite{Beck 1, 137n, Dudek.1, letzter} for upper bounds.
The current best known bounds are $ \big(3.75n-o(1)\big)n\leq \hat{R}(P_n)\leq 74n $ for sufficiently large $n$, where the lower bound is due to Bal and Debiasio \cite{Bal} and the upper bound is obtained by Dudek and Pra\l at \cite{Dudek.1}. 
For the multicolor case, the best known bounds are due to Dudek and Pra\l at \cite{Dudek.2} (see also \cite{Krivel}) who proved that there are positive constants $ c $ and $ C $ such that for every integers $n,r\geq 2$, $ cr^2n\leq \hat{R}(P_n, r)\leq Cr^2(\log r)n $. There are also some pieces of research on the directed version of the problem, see e.g. \cite{sudakov}. \\

In this paper, we focus on the multicolor size-Ramsey number of cycles. The mentioned result by Haxell, Kohayakawa and \L uczak \cite{Haxell-cycle} stating that $ \hat{R}(C_n, r) \leq c_rn $,  uses the regularity lemma and cannot determine how is the dependency of $c_r$ in $r$. Javadi, Khoeini, Omidi and Pokrovskiy \cite{JKhOP}
gave an alternative proof for the linearity of $\hat{R}(C_n,r)$ avoiding the use of the
regularity lemma.
They proved that if $ n $ is even, then $ c_r $ is exponential in $ r $ and if $ n $ is odd, then $ c_r $ is doubly exponential in $ r $. More precisely, they proved that if $ n $ is even, then $ \hat{R}(C_n, r)=O(81^r)n $, and if $ n $ is odd, then $ \hat{R}(C_n, r)=O(35^{2^r})n $, while these bounds are very far from the best known lower bound 
$ \Omega (r^2)n=\hat{R}(P_n, r)\leq \hat{R}(C_n, r) $. 
In this paper, we enlighten the situation and give a more knowledge about the order of magnitude of $c_r$ in $r$.  To be more precise, we prove that there are positive constants $c_1,c_2$ such that for every positive integers $ n,r\geq 2 $, if $n$ is even, then $c_1r^2\,n\leq \hat{R}(C_n,r)\leq  c_2 r^{120} (\log^2 r)\, n$ and if $n$ is odd, then 
$c_12^{r}\, n \leq \hat{R}(C_n, r)\leq  c_2 2^{16 r^2+2\log r}\, n $.
In particular, in the case of even $n$, we actually prove an stronger result. To sate this, we need a definition from \cite{Haxell-tree}. Let a real number $ 0 < \gamma < 1 $  be
fixed, and suppose that $ G $ and $ H $ are arbitrary graphs. We write
$ G \longrightarrow_{\gamma} H $ if any subgraph $ J$ of $ G $ with size $ e(J) \geq  \gamma\, e(G) $ contains an isomorphic copy of $ H $ as a subgraph. It is easy to see that, for every integer $r\geq 2$, if $G\longrightarrow_{1/r} H$ then $G\longrightarrow (H)_r$. We prove that for every integer $r\geq 2$ and every even integer $n$, there exists a graph $G$ on $ O(r^{120} \ln^2 r)\, n$ edges such that $G\longrightarrow_{1/r} C_n$. Note that this result cannot be generalized to the case of odd $n$, since there is no graph $G$ such that $G\longrightarrow_{1/r} C_n$ (see the arguments after the proof of Theorem~\ref{Odd}).
\subsection{Conventions and Notations}
For a graph $ G $, we write $ V(G) $, $ E(G) $ and $ e(G) $ for the vertex set, the edge set and the number of edges of $ G $, respectively. By $|G|$ we mean the number of vertices of $G$. A graph $ G $ with vertex set $ V $ and edge set $ E $ is denoted by $ G=G(V, E) $. 
For $ v \in V(G) $, by $ N_{G}(v) $ we mean the set of all neighbors of $ v $ in $G$ and the degree  of $ v $ is defined as
$ d_{G}(v)=\vert N_{G}(v) \vert $.
For a subset $ X\subseteq V(G) $, we define the neighborhood of $X$ as $ N(X)=N_G(X)=\cup_{x\in X}N_G(x) $. Also, the induced subgraph of $G$ on $X$, denoted by $G[X]$, stands for the graph obtained from $G$ by deleting all vertices in $V(G)\setminus X$.
A rooted tree $T$ with at most two children for each vertex is called the binary tree. 
The maximum distance of any vertex from the root is called the height of $ T $. If a tree has only one vertex (the root), its height is zero.
Moreover, for a tree $ T $ rooted at $ v $, we define $ d_{root}(T)=d_{T}(v) $. \\
%For $ X\subseteq V(G) $,  $ G[X] $ is the induced subgraph of $ G $ with vertex set $ X $. We write $ G\setminus X $ for $ G[V(G)\setminus X] $.
Let $ A,B \subset V(G) $, then $ E_G(A,B)= \{ xy\in E(G) : x\in A, y\in B\} $
is the set of edges connecting a vertex of $ A $ to a vertex of $ B $. Also, $e_G(A,B)=|E_G(A,B)|$. A bipartite graph $ G $ with a bipartition $ (V_0, V_1) $ and the edge set $ E $ is denoted by $ G=G(V_0, V_1, E) $.
%%%%%%%%%%%%%%%%%%%%%%%%%%%%%%%%%%%%%%%%%%%%%%%%%%%%%%%%%%%%%%%%%%%%%%%%
\section{Tools}
In this section, we provide some results that will be used in the follow-up sections.\\
We shall employ the following standard version of
Chernoff's inequality (see, e.g., \cite{janson2011random} {Corollary~2.3}) on the deviation of the binomial random variable~$ Bin(n,p)$. 
\begin{theorem}[Chernoff's inequality]\label{thm:chernoff}
Let $ X $ be a random variable with binomial distribution $ Bin(n,p)$ and the expectation $ \mathbb{E}(X)=\mu $. 
For every $\delta\in(0,3/2]$ we have 
	\begin{align*}
\mathbb{P}\big(\big|X -\mu)\big|>\delta \mu\big) < 2 \exp\big(-\delta^2\mu/3\big).
	\end{align*}
\end{theorem}

Recall that the binomial random
graph $ G(N, p) $ is a distribution over the class of graphs with vertex set $ [N]=\{1, \dots, N\} $ in which every
pair $ \{i, j\} \in \binom{[N]}{2} $ appears independently as an edge in $ G $ with probability $ p=p(N) $, which may (and
usually does) tend to zero as $ N $ tends to infinity.
Similarly, The binomial random bipartite graph $ G(N, N, p) $ is a distribution over the class of bipartite graphs $ G=G(V_0, V_1, E) $ with 
$ |V_0|=|V_1|=N $ in which every pair $ \{i, j\}\in V_0\times V_1 $ appears as an edge independently with probability $ p $. 
Furthermore, we say that an event $ A_N $ in
a probability space holds asymptotically almost surely (or whp), if the probability that $ A_N $ holds tends to $ 1 $ as $ N $ goes to infinity.

In the following lemma, we prove that there exist bipartite graphs whose local densities are similar to the expectation of local densities in random bipartite graphs (for similar results see \cite{krivel}).

\begin{lemma}\label{quasi}
Let $c_1$ be a positive integer and $c_2,c_3,\epsilon, \delta$ be positive numbers such that $c_3\leq c_1$, $0<\epsilon\leq 1/2$ and $3/2\geq \delta >\sqrt{6{\ln({c_1e}/{c_3})/(c_2c_3)}} $. Then, there exists $n_0=n_0(c_1,c_2,c_3,\epsilon, \delta)$, where the following holds for every $n\geq n_0$. \\
Let $ N=c_1n $ and $ p={c_2/n} $. 
There exists a bipartite graph $ G=G(V_0,V_1,E) $, with $ |V_i|=N $, $ i\in\{0,1\} $, such that
\begin{enumerate}
\item \label{quasi 1} $ (1-n^{(\epsilon-1/2)}) pN^2 \leq |E(G)|\leq (1+n^{(\epsilon-1/2)}) pN^2 $,
\item \label{quasi 2} For every two subsets $ U\subseteq V_0 $ and $ W\subseteq V_1 $ with $ |U|=u$ and $ |W|=w$, where $u,w\geq c_3n $, we have
\[
\big| e_G(U,W) -puw\big|\leq \delta puw.
\]
\end{enumerate} 
\end{lemma}
%--------------------------------------------------------------------------------------------
\begin{proof}
Consider the random bipartite graph $ G=G(N, N, p) $ with bipartition $ (V_0,V_1) $ such that $|V_0|=|V_1|=N$ and each edge $v_0v_1$, for $v_0\in V_0,v_1\in V_1$, exists independently with probability $p$, where $N,p$ are as in the statement of the lemma. We prove that a.a.s. $G$ satisfies Properties \ref{quasi 1} and \ref{quasi 2}. First assume that $ U\subseteq V_0 $ and $ W\subseteq V_1 $ are two fixed subsets with $ u=|U|\geq c_3n $ and $ w=|W|\geq c_3n $. Let $ X_{U,W} $ be the random variable that counts the number of edges between $ U $ and $ W $, i.e. $X_{U,W}=e_G(U,W)$. Clearly, $ X_{U,W} $ has binomial distribution $ Bin(uw, p) $ and its expectation is $ \mathbb{E}(X_{U,W})= puw $. By applying Chernoff's inequality in Theorem \ref{thm:chernoff}, the probability that Property \ref{quasi 2} fails to be hold for a fixed choice of $ U $ and $ W $ is bounded from above as follows.
\begin{align*}
	\mathbb{P}\big(\big| X_{U,W} -puw\big|> \delta puw\big) &\leq 2e^{-\dfrac{\delta ^2 puw}{3}} = 2e^{-\dfrac{\delta ^2 c_2 uw}{3n}}.
\end{align*}	
Now, by taking union bound on all choices of $ U $ and $ W $ and using inequality $ \binom{N}{k}\leq (\frac{Ne}{k})^k $, we arrive at 
\begin{align*}
\mathbb{P}(\text{Property \ref{quasi 2} fails}) &\leq  \sum_{u= \lceil c_3n \rceil}^{N}\sum_{w= \lceil c_3n \rceil}^{N}\binom{N}{u}\binom{N}{w}2e^{-\dfrac{\delta ^2 c_2 uw}{3n}}\\
&\leq \sum_{u= \lceil c_3n \rceil}^{N}\sum_{w= \lceil c_3n \rceil}^{N} \big(\frac{Ne}{u}\big)^u\big(\frac{Ne}{w}\big)^w 2e^{-\dfrac{\delta ^2 c_2 uw}{3n}}\\
&=2\sum_{u= \lceil c_3n \rceil}^{N}\sum_{w= \lceil c_3n \rceil}^{N} e^{h(\frac{u}{n}, \frac{w}{n})n},
\end{align*}
where 
\[
h(x, y):=  x \ln(\frac{c_1e}{x})+y \ln(\frac{c_1e}{y})-\frac{\delta ^2}{3}c_2 xy.
\]
One can verify that the function $ h(x,y) $ is decreasing for both $x\in [c_3,c_1]$ and $y\in [c_3,c_1]$. To see this, note that
\[\frac{\partial h}{\partial x} = \ln (\frac{c_1}{x})-\frac{\delta ^2}{3}c_2y \leq \ln (\frac{c_1}{c_3})-\frac{\delta ^2}{3}c_2c_3<0,
\]
which is the case by the choice of $\delta$.
Similarly, $\partial h/\partial y<0$ and therefore, we have 
\[
\mathbb{P}\big(\text{Property} ~ \ref{quasi 2} ~\text{fails}\big) \leq 2N^2 e^{h(c_3,c_3)n}=2c_1^2 n^2 e^{h(c_3,c_3)n}.
\]
 Now, if we choose $ \delta $ so that
$ \delta ^2>6\dfrac{\ln({c_1e}/{c_3})}{c_2c_3}  $, then $ h(c_3, c_3)<0 $ and we have
\[
\mathbb{P}\big(\text{Property} ~ \ref{quasi 2} ~\text{fails}\big)=o(1).
\]
In order to see Property \ref{quasi 1}, note that $ e(G) $ has binomial distribution $ Bin(N^2, p) $ with expectation $ \mathbb{E}(e(G))=pN^2=c_1^2 c_2n $.
Again, applying Chernoff's inequality, we have 
\begin{align*}
\mathbb{P}\big(\text{Property} ~ \ref{quasi 1} ~\text{fails}\big)=\mathbb{P}\big( |e(G)-pN^2|> n^{(\epsilon-1/2)} pN^2\big)&\leq e^{-\frac{  n^{(2\epsilon-1)} pN^2}{3}}\\
&=e^{-\dfrac{n^{2\epsilon} c_1^2 c_2}{3}}= o(1).
\end{align*}
Thus, again by taking union bound, the probability that Properties \ref{quasi 1} or \ref{quasi 2} fail is in $o(1)$. This completes the proof.
\end{proof}
%%%%%%%%%%%%%%%%%%%%%%%%%%%%%%%%%%%%%%%%%%%%%%%%%%%%%%%%%%%%%%%%%%%%%%%%%%%%%%%%%%%%%%%%%%%%%%%%%%%%%%%%%%%%%%%%%%%%
The following lemma is the non-bipartite counterpart of Lemma~\ref{quasi}. The proof is similar and so we  omit it.
\begin{lemma}\label{quasi-graph}
	Let $c_1$ be a positive integer and $c_2,c_3,\epsilon, \delta$ be positive numbers such that $2c_3\leq c_1$, $0<\epsilon<1$ and $ 3/2\geq \delta >\sqrt{6{\ln(\frac{c_1e}{c_3})/(c_2c_3)}} $. Then, there exists $n_0=n_0(c_1,c_2,c_3,\epsilon, \delta)$, where the following holds for every $n\geq n_0$. \\
	Let $ N=c_1n $ and $ p={c_2/n} $. 
	There exists a graph $ G=G(V,E) $, with $ |V|=N $, such that
	\begin{enumerate}
		\item \label{quasi 1} $ (1-\epsilon) \frac{c_1 c_2 (c_1n-1)}{2}\leq |E(G)|\leq (1+\epsilon) \frac{c_1 c_2 (c_1n-1)}{2}$,
		\item \label{quasi 2} For every two disjoint subsets $ U,W\subseteq V $ with $ |U|=u$ and $ |W|=w$, where $u,w\geq c_3n $, we have
		\[
		\big| e_G(U,W) -puw\big|\leq \delta puw.
		\]
\end{enumerate}   
\end{lemma}
%%%%%%%%%%%%%%%%%%%%%%%%%%%%%%%%%%%%%%%%%%%%%%%%%%%%%%%%%%%%%%%%%%%%%%%%%%%%%%%%%%%%%%%%%%%%%%%%%%%%%%%%%%%%%%%%%%%%
Extending a tree-embedding result of Friedman and Pippenger \cite{Friedman}, Balogh, Csaba and Samotij \cite{Balogh}
proved that all graphs satisfying certain expansion properties contain all trees with bounded maximum degree and certain number of vertices.
\begin{theorem}\label{balogh} {\rm \cite{Balogh}}
Let $ D,m_0,M_0,m_1 $ and $ M_1 $ be positive integers. Assume that $ H $ is a non-empty bipartite graph with bipartition $ (V_0,V_1) $ satisfying the following conditions.
\begin{enumerate}
\item For every $ X \subseteq V_i $ with $ 0<|X| \leq m_i $, $ |N_H(X)| \geq D|X| + 1 $ for $ i \in\{0,1\} $.
\item For every $ X \subseteq V_i $ with $ m_i < |X| \leq 2m_i $, $ |N_H(X)| \geq D|X| + M_{1-i} $ for $ i \in\{0,1\} $.
\end{enumerate}
Then $ H $ contains every tree $ T $ with the maximum degree at most $ D $ and the bipartition with  parts of sizes $ M_0 $ and $ M_1 $.
\end{theorem}
%%%%%%%%%%%%%%%%%%%%%%%%%%%%%%%%%%%%%%%%%%%%%%%%%%%%%%%%%%%%%%%%%%%%%%%%%%%%%%%%%%%%%%%%%%%%%%%%%%%%%%%%%%%%%%%%%%%%
The above theorem can be generalized to forests. The following theorem enables us to embed a bounded degree forest with prescribed roots into a graph with an appropriate expansion property. 
This theorem is an extension of Theorem \ref{balogh} and is indeed a modified version of Lemma $ 3.1 $ in \cite{Ramsey good} with a similar proof. However, for the sake of completeness, we include a proof in Appendix \ref{app}.
\begin{theorem}\label{forst-bipartite}
	Let $ \Delta, M, t $ and $ m $ be given positive integers. Let $ G=G(V_0, V_1, E) $ be a bipartite graph and $ X=\{x_1, \dots, x_t\} $ be a subset of $ V_0\cup V_1 $. Also, let $ T_{x_1}, \dots, T_{x_t} $ be $t$ rooted trees satisfying $ \sum_{i=1}^{t}|T_{x_i}|\leq M $ and $ \Delta(T_{x_i})\leq \Delta $ for all $ i\in\{1,\ldots ,t\} $. Suppose that we have the following properties in $ G $. 
	\begin{enumerate}
		\item For all $ S\subseteq V_i $, $ i\in \{0, 1\} $, with $ |S|\leq m $ we have 
		\begin{align}\label{A.1}
		|N_G(S)\setminus X|\geq 2\Delta |S\setminus X|+ \sum_{x\in S\cap X}(d_{root}(T_x)+\Delta),
		\end{align}
		\item for all $ S\subseteq V_i $, $ i\in \{0, 1\} $,  with $ m\leq |S|\leq 2m $ we have $ |N_G(S)|\geq M+8\Delta m $.
	\end{enumerate}
	Then we can find disjoint copies of the trees $ T_{x_1}, \dots, T_{x_t} $ in $ G $ such that for each $ i $, $ T_{x_i} $ is rooted at $ x_i $.  
\end{theorem}
%\textcolor{red}{we can write it with $m_i$ and $M_i$}
%%%%%%%%%%%%%%%%%%%%%%%%%%%%%%%%%%%%%%%%%%%%%%%%%%%%%%%%%%%%%%%%%%%%%%%%%%%%%%%%%%%%%%%%%%%%%%%%%%%%%%%%%%%%%%%
\section{Multicolor size-Ramsey number of even cycles}
In this section, we prove that for even integer $n$, $\hat{R}(C_n,r)= O(r^{120}\log^2 r)n$. Indeed, we prove a stronger result which states that there exists a graph $G$ with $ O(r^{120}\log^2 r)n $ number of edges such that $G\to_{1/r} C_n$ (for the definition, see the introduction).

Our approach is to prove that every bipartite graph with an appropriate connectivity characteristic fulfills an expansion property and thus contains an even cycle. Then, we prove that in every random bipartite graph $G$, whp every dense subgraph of $G$ has the mentioned connectivity characteristic and this leads us to the main result.  First, we need a definition.
\begin{definition}
	Let $ G = (V_0, V_1 ,E) $ be a bipartite graph with  $ |V_0|=|V_1|=N $ vertices on each part and let $ \alpha $ be a positive integer. The bipartite graph $ G $ is called an
	 $\alpha$-joined graph if for every pair of vertex subsets $ A\subseteq V_0 $ and $ B\subseteq V_1 $ with sizes $ |A|, |B| \geq \alpha N  $, we have $e(A,B)\neq 0$.
\end{definition}

The following lemma  shows that every $\alpha$-joined bipartite graph $ G $ contains an induced subgraph $ G' $ with an strong expansion property. A similar result has been proved in \cite{Krivel}.
\begin{lemma}\label{expander}
Let $0<\alpha <1 $ and $ G=(V_0, V_1, E) $ be an $ \alpha $-joined bipartite graph with $ |V_0|=|V_1|=N $. Then $ G $ contains an induced subgraph $ G'=(V_0', V_1', E') $ that satisfies the following properties.
\begin{enumerate}
\item $ |V_0'|, |V_1'|\geq (1-\alpha)N $.
\item For every $ U \subseteq V_i' $, $ i \in\{0, 1\} $, with $ 0<|U| \leq \alpha N$, we have $ |N_{G'}(U)| > \frac{1-2\alpha}{2\alpha}|U| $.
\item For every $ U \subseteq V_i' $, $ i \in\{0, 1\} $, with $  |U|>\alpha N $, we have $ |N_{G'}(U)| > (1-2\alpha)N $.
\end{enumerate}
\end{lemma} 
\begin{proof}
Start with $ G':=G $, $ V'_0:=V_0 $ and $ V'_1:=V_1 $. As long as there is a vertex set $ U\subseteq V_i $, $ i\in\{0 ,1\} $, of size $ |U|\leq \alpha N $ with 
$ \big| N_{G'}(U)\big|\leq \frac{1-2\alpha}{2\alpha}|U| $, delete $ U $ from $ V_i $ and update $ G':=G'[V(G')\setminus U] $ and $V'_i=V_i\setminus U$. Let $ V''_0=V_0\setminus V'_0 $ and 
$ {V''_1}=V_1 \setminus V'_1 $ be the union of all deleted subsets in $ V_0 $ and $ V_1 $, respectively. We prove the following claim.
\begin{claim}\label{clm1}
For $ i\in \{0, 1\} $, we have $ |V''_i|\leq \alpha N $.
\end{claim}
In order to prove Claim \ref{clm1}, note that at each iteration of the above process, a set of size at most $ \alpha N $ is added to $ V''_0 $ or $V''_1$. Now, by the contrary, suppose that in an iteration either $| V''_0| $ or $|V''_1|$, say $|V''_0|$, exceeds $\alpha N$. So,  $ \alpha N < |V''_0|\leq 2\alpha N $ and $|V''_1|\leq \alpha N$.  
Also, note that at any iteration, a set $U$ with $ \big| N_{G'}(U)\big|\leq \frac{1-2\alpha}{2\alpha}|U| $ is added to $V''_0$. Thus, 
\begin{align*}
|N_G(V''_0)|\leq \dfrac{1-2\alpha}{2\alpha}|V''_0|+|V''_1|&\leq \dfrac{1-2\alpha}{2\alpha} 2 \alpha N +\alpha N\\
&= (1-\alpha )N.
\end{align*} 
Therefore, $ |V_1\setminus N_{G}(V''_0)|\geq |V_1|-(1-\alpha)N = \alpha N $. Since $ G $ is an $ \alpha $-joined graph, then we should have an edge between $ V''_0 $ and $ V_1\setminus N_G(V''_0) $, which is a contradiction. This proves Claim~\ref{clm1}.

Therefore, we have $ |V_i'|\geq (1-\alpha)N $, $ i\in \{0, 1\} $, and for every subset $ U\subseteq V_i' $ with $ |U|\leq \alpha N $, we have $ |N_{G'}(U)|> \frac{1-2\alpha}{2\alpha}|U| $. Moreover, since $G$ is an $\alpha$-joined graph, so is $G'$. Thus, for every $U\subset V'_i$, with $|U|\geq \alpha N$, we have $|V'_{1-i}\setminus N_{G'}(U)|< \alpha N$. Hence, $|N_{G'}(U)|> |V'_{1-i}|-\alpha N\geq (1-2\alpha)N$. This proves that $G'$ satisfies conditions (1)-(3).
\end{proof}

%%%%%%%%%%%%%%%%%%%%%%%%%%%%%%%%%%%%%%%%%%%%%%%%%%%%%%%%%%%%%%%%%%%%%%%%%%%%%%%%%%%%%%%%%%%
In the following lemma, we prove that every $ \alpha $-joined bipartite graph, with a suitable choice of $ \alpha $, contains a large even cycle.
\begin{lemma}\label{joined graph}
Let $ n $ be a positive even integer and $\alpha$ be a number with $0<\alpha <0.1 $. Then, every $ \alpha $-joined bipartite graph 
$ G=G(V_0, V_1, E) $ with $ |V_i|=N\geq (n+4)/(2-20\alpha)$, $ i\in \{0, 1\} $, contains a copy of $ C_n $.
\end{lemma}
\begin{proof}
Since $ G $ is an $ \alpha $-joined bipartite graph, using Lemma \ref{expander}, G contains an induced subgraph $ G'=G'(V_0', V_1', E') $ with at least 
$ (1-\alpha)N $ vertices in each part such that for every vertex set $ U\subseteq V_i' $, $ i\in \{0, 1\} $, with  $ |U|\leq \alpha N $ we have $ |N_{G'}(U)|\geq \frac{1-2\alpha}{2\alpha}|U| $. Moreover, if $  |U|> \alpha N$, then we have $ |N_{G'}(U)|\geq (1-2\alpha)N $.

Now let $ T_0 $ and $T_1$ be two disjoint copies of a binary tree with $ \lceil \alpha N \rceil $ leaves and height $ \lceil \log \lceil \alpha N\rceil \rceil $ 
and at most $ 2 \lceil \alpha N\rceil$ vertices. For each $i\in \{0,1\}$, let $X_i$ be the set of leaves of $T_i$. Also, let $ T $ be a tree on at most $ 4\lceil\alpha N\rceil +n$ vertices formed by attaching the roots of $ T_0 $ and $ T_1 $ by a path of length $ n-1-2\lceil \log \lceil \alpha N\rceil \rceil $ (see Figure~\ref{fig1}). Note that $ T $ has the maximum degree $ 3 $ and for every pair of vertices $ x_0\in X_0 $ and $ x_1\in X_1 $, there is a path of length $ n-1 $ in $T$ from $ x_0 $ to $ x_1 $. 

We are going to prove that $G'$ contains a copy of $ T $. For this purpose, we apply Theorem~\ref{balogh} with setting $D=3$, $ m_0=m_1= \lfloor\alpha N\rfloor $ and $ M_0=M_1=2\lceil \alpha N\rceil +\frac{n}{2} $. So, by Lemma~\ref{expander} and Theorem~\ref{balogh}, if the following two conditions hold, then $G'$ contains a copy of $T$. 
\begin{enumerate}
\item $ \dfrac{1-2\alpha}{2\alpha}\geq 3 $,
\item $ (1-2\alpha) N \geq 6\alpha N+M_i=6\alpha N+ 2\lceil \alpha N\rceil +\frac{n}{2} $.
\end{enumerate}
Hence, it suffices that $\alpha <1/8$ and $2N(1-10\alpha)\geq n+4$ which hold by assumptions. Therefore, $ G' $ contains a copy of $ T $. 
Note that, for every vertices $x_0\in X_0 $ and $x_1\in X_1$, there is a path $P_{x_0x_1}$ of length $ n-1 $ in $T$ from $ x_0 $ to $ x_1 $. Thus, since $n$ is even, $X_0$ and $X_1$ are in different parts of $G$. Now, since $ G $ is an $ \alpha $-joined bipartite graph and $|X_0|,|X_1|\geq \alpha N$, there is an edge $ e=x_0x_1\in E(G) $ for some $ x_i\in X_i $, $i\in \{0,1\}$.  The path $P_{x_0x_1}$ along with the edge $ e $ forms a cycle $ C_n $ in $G$ (see Figure~\ref{fig1}). This completes the proof. 
\begin{figure}
{\includegraphics[width=\textwidth]
	       {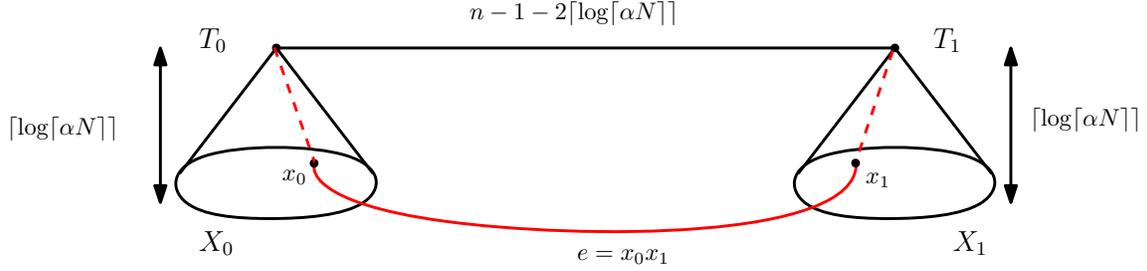}}
%	\vspace{3in}
\caption{The tree $T$ which leads to the existence of a cycle $C_n$.}\label{fig1}
\end{figure}
\end{proof}

Now, we are ready to prove the main result of this section. 
%%%%%%%%%%%%%%%%%%%%%%%%%%%%%%%%%%%%%%%%%%%%%%%%%%%%%%%%%%%%%%%%%%%%%%%%%%%%%%%%%%%%%%%%%%%
\begin{theorem}\label{main}
Let $ r\geq 2 $ be an integer. For every sufficiently large even integer $ n $, there exists a bipartite graph $G$ with at most $O(r^{120} \ln^2 r)n$ edges such that $G\to_{1/r} C_n$.
\end{theorem}
\begin{proof}
Let $d,f$ be two positive integers. Apply Lemma \ref{quasi} by setting the parameters $ c_1=r^d$, $c_2=r^f \ln^2 r $, $c_3= 6/11$, $\epsilon=1/4$ and 
$ \delta =\sqrt{11\ln(r^{d+3})/(r^f \ln^2 r)}=\sqrt{11(d+3)/(r^f \ln r)}$. Note that $\delta > \sqrt{6{\ln({c_1e}/{c_3})/(c_2c_3)}} $, because $r^3> 11e/6$. Also, suppose that $f,d$ are chosen such that $\delta <1$. Thus, conditions of Lemma~\ref{quasi} hold and there exists a bipartite graph $ G=G(V_0,V_1,E) $, where $|V_0|=|V_1|=N=r^d n$ and $p=r^f \ln^2 r/n$, satisfying Properties (1) and (2) in Lemma~\ref{quasi}. 

%Also, suppose that $f,d$ are chose such that there is a number $\delta$ satisfying,
%\begin{align}
%\frac{3}{2} \geq \delta > \sqrt{\dfrac{11\ln({r^d11e}/{6})}{r^f \ln^2 r}}.
%\end{align}

 Now, let $H$ be a spanning subgraph of $G$ with $e(H)\geq e(G)/r$. We are going to prove that $H$ contains a copy of $C_n$.

By the contrary, suppose  that there is no copy of $ C_n $ in $ H $. Let $F$ be the graph obtained from $G$ by removing all edges of $H$. By counting $ e(F) $, we will reach to a contradiction. 
First we need the following result.
%\begin{claim}\label{clm2}
%For every vertex set $V_i'\subseteq V_i$, $i\in\{0,1\}$, with $|V_0'|=|V_1'|=N'$, then 
%$ e_F(V_0',V_1')\geq (1-\delta) \big(1-(\frac{n+4}{(2-20\alpha)N'})^c \big) pN'^2.$ {\textcolor{red}{Shouldn't we say s.th about "c" here?}} %=(1-\delta)p N^2-(1-\delta)p (150n)^c N^{2-c}$.
%\end{claim}
%We prove Claim \ref{clm2} by induction on $ N' $. First note that if $N'\leq (n+4)/(2-20\alpha)$, then the right hand of the above inequality is non-positive and so the claim obviously holds.  
%
%Now, suppose that $ N'> (n+4)/(2-20\alpha) $. Let $H'$ be the induced subgraph of $H$ on $V_0'\cup V_1'$. Since there is no copy of $ C_n $ in $ H' $, using Theorem~\ref{joined graph}, $ H'$ is not an $ \alpha $-joined graph, i.e. there are two subsets $ U\subseteq V'_0 $ and $ W\subseteq V'_1 $ both with size $ \lceil \alpha N'\rceil $ such that $e_{H'}(U,W)=0$. Thus, using Lemma \ref{quasi}, we have
\begin{claim}\label{clm2}
	Let $n$ be sufficiently large and $N'$ be a positive integer. Then for every vertex set $V_i'\subseteq V_i$, $i\in\{0,1\}$, with $|V_0'|=|V_1'|= N'$, we have  
	\begin{equation} \label{claimF}
	e_F(V_0',V_1')\geq (1-\delta) \left(1-\left(\frac{7n}{N'}\right)^\lambda \right) pN'^2,
	\end{equation}
where $\lambda=.017$.
\end{claim}
We prove Claim \ref{clm2} by induction on $ N' $. Let $V_i'\subseteq V_i$, $i\in\{0,1\}$, be such that $|V_0'|=|V_1'|= N'$. First, note that if $N'\leq {7n}$, then the right hand of Inequality~\ref{claimF} is non-positive and so the claim obviously holds.  

Now, suppose that $ N'> 7n $. Also, let $N'=11q+w	$, where $0\leq w\leq 10$. Let $H'$ be the induced subgraph of $H$ on $V_0'\cup V_1'$. Since there is no copy of $ C_n $ in $ H' $, using Lemma~\ref{joined graph} with $\alpha=1/11$ and $n\geq 15$, $ H'$ is not an $ \alpha $-joined graph, i.e. there are two subsets $ U\subseteq V'_0 $ and $ W\subseteq V'_1 $ both with size $ \lfloor \alpha N'\rfloor=q $ such that $e_{H'}(U,W)=0$. Thus, using Lemma \ref{quasi}, we have
\[
e_F(U, W)=e_G(U,W)\geq  (1-\delta)p q^2.
\]
%where, $3/2\geq  \delta >\sqrt{11\ln(r^d11e/6)/(r^f \ln^2 r)} $.\\
% Set $ X:=V'_0\setminus U $ and $ Y:=V'_1\setminus W %First, suppose that $\alpha N$ is an integer (we will show how we can handle the case that it is not a whole number).
  In order to apply the induction hypothesis, let $X=V_0'\setminus U$ and $Y=V_1'\setminus W$ and let $ X_1,\ldots, X_{10} $ and $ Y_1,\ldots, Y_{10}$ be disjoint subsets of $ X $ and $Y$, respectively, such that for each $ i\in \{1,\ldots,10 \} $ we have 
$ |X_i|=|Y_i|= q$.\\
 Now, by the induction hypothesis, we have 
\begin{align*}
e_F(V_0',V_1') &\geq  e_F(U, W)+\sum_{i=1}^{10} (e_F(U, Y_i)+e_F(W, X_i))+e_F(X, Y)\\
&\geq (1-\delta)p \bigg[q^2 + 20 \left(1-\left(\frac{7n}{q}\right)^\lambda\right) q^2+\left(1-\left(\frac{7n}{(N'-q)}\right)^\lambda\right) (N'-q)^2\bigg]\\
&= (1-\delta)p \bigg[21q^2 +(N'-q)^2 - 20 ({7n})^\lambda q^{2-\lambda}-(7n)^\lambda (N'-q)^{2-\lambda}\bigg]\\
&= (1-\delta)p \bigg[N'^2-2qw - 20 (7n)^\lambda q^{2-\lambda}-(7n)^\lambda (N'-q)^{2-\lambda}\bigg]\\
&\geq (1-\delta)p \bigg[N'^2-20(7n)^\lambda q - 20 (7n)^\lambda q^{2-\lambda}-(7n)^\lambda (N'-q)^{2-\lambda}\bigg]\\
&= (1-\delta)p \bigg[N'^2-(7n)^\lambda\big( 20q + 20 q^{2-\lambda}+ (N'-q)^{2-\lambda}\big)\bigg]\\
&\geq (1-\delta)p \bigg[N'^2-(7n)^\lambda\big( 20N'/11 + 20 (N'/11)^{2-\lambda}+(N'-N'/11)^{2-\lambda}\big)\bigg],
\end{align*}
where the last inequality yields from the fact that $q\leq N'/11$ and the function $20q + 20 q^{2-\lambda}+ (N'-q)^{2-\lambda}$ is increasing with respect to $q$ whenever $0<\lambda<1$. Therefore, 
\begin{align*}
e_F(V_0',V_1') 
&\geq (1-\delta)p \bigg[N'^2-(7n)^\lambda N'^{(2-\lambda)}\big( 20N'^{(\lambda-1)}/11 + 20 (\frac{1}{11})^{2-\lambda}+(\frac{10}{11})^{2-\lambda}\big)\bigg], \\
&\geq (1-\delta)p \bigg[N'^2-(7n)^\lambda N'^{(2-\lambda)}\big( 20(7n)^{(\lambda-1)}/11 + 20 (\frac{1}{11})^{2-\lambda}+(\frac{10}{11})^{2-\lambda}\big)\bigg],
\end{align*}
where the last inequality holds since $N'\geq 7n$ and $0<\lambda<1$.
Now, in order to complete the proof of the claim, it suffices to have 
\begin{align}\label{aa}
20(7n)^{(\lambda-1)}/11+ 20 (1/11)^{2-\lambda}+(10/11)^{2-\lambda}\leq 1.
\end{align}
which is the case for $\lambda=.017$ and sufficiently large $n$. This completes the proof of Claim~\ref{clm2}.

\paragraph{}

Hence, by Claim~\ref{clm2}, we have 

\begin{equation}
\label{eq:lower}
e(F) \geq (1-\delta)p(1-\frac{7^\lambda}{r^{d\lambda}})N^2.
\end{equation}
On the other hand, $ e(H)\geq e(G)/r $, so by Lemma~\ref{quasi}, we have 
\begin{align}\label{bb}
e(F)=e(G)-e(H)\leq (1-\frac{1}{r})e(G)\leq (1-\frac{1}{r})(1+n^{-1/4}) pN^2.
\end{align}
Combining \eqref{eq:lower} and \eqref{bb} yields,
\begin{equation} \label{eq:lu}
(1-\delta)p(1-\frac{7^\lambda}{r^{d\lambda}})N^2 \leq e(F) \leq (1-\frac{1}{r})(1+n^{-1/4})pN^2.
\end{equation}
Now, with $ \delta =\sqrt{11(d+3)/(r^f \ln r)}$, we have
\begin{align}\label{ineq.A}
\left(1-\sqrt{\frac{11(d+3)}{r^f \ln r}}\right)\left(1-\frac{7^\lambda}{r^{d\lambda}}\right)  \leq \left(1-\frac{1}{r}\right)\left(1+\frac{1}{n^{1/4}}\right).
\end{align}
First, suppose that $r$ is sufficiently large and set $ d=59 $ and $ f=2 $, then we have 

\begin{align}
1-\frac{\sqrt{682}}{r\sqrt{\ln r}}- \frac{7^{0.017}}{r^{1.003}}  \leq \left(1-\frac{1}{r}\right)\left(1+\frac{1}{n^{1/4}}\right),
\end{align}
which leads to a contradiction for sufficiently large values of  $r$ and $n$.
This contradiction implies that there should be a copy of $ C_n $ in $ H $. Therefore we have $ G\to_{1/r} C_n $ and moreover, by Lemma~\ref{quasi},
\[
|E(G)|\leq (1+n^{-1/4}) r^{2d+f}(\ln^2 r) n= O(r^{120} \ln^2 r)n.
\]
Note that if $r$ is not large, then one can set $d=62$ and $f=28$ and similar arguments imply the assertion.
\end{proof}

\begin{corollary}
For every integer $r\geq 2$ and  positive even integer $n$, we have $\hat{R}(C_n,r)=  O(r^{120} \ln^2 r)n$.
\end{corollary}
\begin{proof}
If $n$ is sufficiently large, then the result immediately follows from Theorem~\ref{main}. Now, suppose that $n$ is bounded by a constant number $n_0$. Then, by a result in \cite{Lucss}, we have $R(C_n,r)=O(rn)$ and thus, $\hat{R}(C_n,r)\leq \binom{R(C_n,r)}{2}\leq O(r^2n^2)= O(r^2)$.
\end{proof}
%%%%%%%%%%%%%%%%%%%%%%%%%%%%%%%%%%%%%%%%%%%%%%%%%%%%%%%%%%%%%%%%%%%%%%%%%%%%%%%%%%%%%%%%%%%
\section{Multicolor size-Ramsey number of odd cycles}
In this section, we deals with the multicolor size-Ramsey number of odd cycles. 
Javadi et al. \cite{JKhOP} proved that if $ n $ is odd, then $ \hat{R}(C_n, r)\leq35^{2^r}n $, while it is far from the best known lower bound 
$ \Omega (r^2)n=\hat{R}(P_n, r)\leq \hat{R}(C_n, r) $. In this section, we improve both lower and upper bounds to exponential ones in $r$. 
Although, in the previous section we proved that for even $ n $, dependency of $\hat{R}(C_n, r)$ on $ r $ is polynomial, surprisingly here we prove that  $2^{O(r)}n\leq \hat{R}(C_n, r)\leq 2^{O(r^2)}n$.   
First, we prove the lower bound. For this purpose, we need the following two results. 
It is a well-known result of Erd\H{o}s \cite{Erd. bipartite} that every graph contains a large bipartite subgraph.
%(e.g. see Theorem 1.3.19 in \cite{west})
\begin{lemma}\label{bipartite} {\rm \cite{Erd. bipartite}}
Every simple graph $ G $ has a spanning bipartite subgraph with at least  $ e(G)/2 $ edges.
\end{lemma}
\noindent The following theorem, due to Bal and DeBiasio \cite{Bal}  states a lower bound for $ \hat{R}(P_n) $.
\begin{theorem}\label{SRP}{\rm \cite{Bal}}
For every integer $ n\geq 2 $,  we have $ \hat{R}(P_n)\geq (3.75-o(1))n $. 
\end{theorem}
\noindent Using the above results, we may prove an exponential lower bound in $r$ for $ \hat{R}(C_n, r) $, when $ n $ is odd.
\begin{theorem}\label{Odd}
Let $ n,r\geq 2 $ be two integers. If $G$ is a non-bipartite graph containing the path $P_n$, then 
\[
 \hat{R}(G, r)> (3.75-o(1)) 2^{r-2}n.
\]
In particular, the lower bound holds for $G=C_n$, when $n$ is odd.
\end{theorem}
\begin{proof}
Let $H$ be a graph with  at most $ (3.75-o(1)) 2^{r-2}n$ number of edges.
We use induction on $ r $  to show that there is an $ r $-edge coloring of $ E(H) $ with colors $ 1, \dots, r $ such that $ H $ contains no monochromatic copy of $ G $.
First suppose that $ r=2 $. Since $G$ contains $P_n$, using Theorem \ref{SRP}, we have
\[
\hat{R}(G)\geq  \hat{R}(P_n)\geq (3.75-o(1))n.
\]
Therefore, $H$  admits a $ 2 $-edge coloring without a monochromatic copy of $G$.\\
Now, let $ r\geq 3 $ and suppose that $ H $ is a graph with $ e(H)\leq  (3.75-o(1)) 2^{r-2}n $.
Using Lemma~\ref{bipartite}, $ H $ has a spanning bipartite subgraph $ F $ with $ e(F)\geq e(H)/2 $.
Color all edges of $F$ by color $r$. Also, note that $e(H\setminus F)\leq e(H)/2\leq (3.75-o(1)) 2^{r-3}n $. Thus, by the induction hypothesis, there is an $ (r-1) $-edge coloring of $ H\setminus F $ with colors $ 1, \dots, r-1 $ such that $ H\setminus F $ contains no monochromatic copy of $ G $. Furthermore,
the graph $ F $ does not contains $G$, since $G$ is non-bipartite. Hence, $H$ admits an $r$-edge coloring without a monochromatic copy of $G$ and this completes the proof.
\end{proof}
%%%%%%%%%%%%%%%%%%%%%%%%%%%%%%%%%%%%%%%%%%%%%%%%%%%%%%%%%%%%%%%%%%%%%%%%%%%%%%%%%%%%%%%%%%%%%%%%%%%%%%%%%%%%%%%%%%%
%\begin{lemma}{\rm \cite{Ramsey good}}\label{Forst}
%Let $ \Delta, M, t $ and $ m $ be given. Let $ X=\{x_1, \dots, x_t\} $ be a set of vertices in a graph $ G $. Suppose that we have rooted trees $ T_{x_1}, \dots, T_{x_i} $ satisfying $ \sum_{i=1}^{t}|T_{x_i}|\leq M $ and $ \Delta(T_{x_i})\leq \Delta $ for all $ i $. Suppose that for all $ S $ with $ m\leq |S|\leq 2m $ we have $ |\Gamma (S)|\geq M+10\Delta m $, and for $ S $ with $ |S|\leq m $ we have 
%\[
%|\Gamma (S)\setminus X|\geq 4\Delta |S\setminus X|+ \sum_{x\in S\cap X}(d_{root}(T_x)+\Delta).
%\]
%Then we find disjoint copies of the trees $ T_{x_1}, \dots, T_{x_i} $ in $ G $ such that for each $ i $, $ T_{x_i} $ is rooted at $ x_i $. 
%\end{lemma}
%%%%%%%%%%%%%%%%%%%%%%%%%%%%%%%%%%%%%%%%%%%%%%%%%%%%%%%%%%%%%%%%%%%%%%%%%%%%%%%%%%%%%%%%%%%%%%%%%%%%%%%%%%%%%%%%%%%
%%%%%%%%%%%%%%%%%%%%%%%%%%%%%%%%%%%%%%%%%%%%%%%%%%%%%%%%%%%%%%%%%%%%%%%%%%%%%%%%%%%%%%%%%%%%%%%%%%%%%%%%%%%%%%%%%%%
Now, we delve into the upper bound. In Theorem~\ref{main}, we proved that for even integer $n$ and every integer $r\geq 2$, there is a graph $G$ with $e(G)\leq O(r^c) n $ for some constant $c$, where  $ G\longrightarrow_{1/r} C_n $. One may ask if it is also true for odd $n$. Nevertheless, it can be seen that when $n$ is odd there is no graph $G$ with $G\longrightarrow_{1/r} C_n$. To see this, note that for every graph $G$, by Lemma~\ref{bipartite}, $G$ contains a bipartite spanning subgraph $F$ with $e(F)\geq e(G)/2\geq e(G)/r$, while $F$ contains no odd cycle. Thus, $G\not\longrightarrow _{1/r} C_n$.
In the light of this observation, one may see that the approach of the proof of Theorem~\ref{main} does not work to prove an upper bound for $\hat{R}(C_n,r)$ when $ n $ is odd and we have to deploy a new tool.

The following lemma, is the counterpart of Lemma~\ref{joined graph} for odd cycles. Here, we show that if $ G $ is a graph without an odd cycle $ C_n $, then either it contains two large disjoint sets without any edge between them, or it contains two large disjoint stable sets.
\begin{lemma}\label{Beta-odd}
Let $ n $ be a positive odd integer and $ G $ be a graph which contains no copy of $C_n$. Also, let $\alpha$ be a number, where $ 0<\alpha \leq 0.033 $. Then, for every two disjoint subsets $ V_0 $ and $ V_1 $ of vertices with $ |V_0|=|V_1|=N\geq ({n+28})/(1-30\alpha) $, either 
\begin{itemize}
	\item  [{\rm (i)}]  there are two subsets $ V_0'\subseteq V_0 $ and $ V_1'\subseteq V_1$ with $ |V_0'|=|V_1'|\geq\alpha N $ and $ e_G(V_0', V_1')=0 $, or
	\item [{\rm (ii)}]  there are two subsets $ V_0'\subseteq V_0 $ and $ V_1'\subseteq V_1$ with $ |V_0'|=|V_1'|\geq(1-\alpha) N$ and
	$  e_G(V_0')=e_G(V_1')=0 $.\\
\end{itemize}
\end{lemma}
\begin{proof}
	\begin{figure}
		{\includegraphics[width=\textwidth]
			{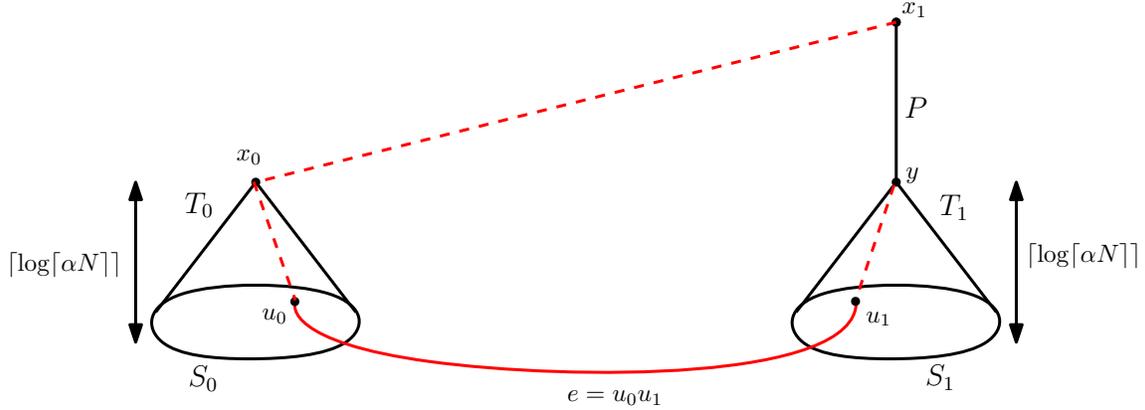}}
		%	\vspace{3in}
		\caption{The trees $T_0$ and $ T_1 $ which lead to the existence of a cycle $C_n$.}\label{fig2}
	\end{figure}
Let $ G $ be a graph which contains no copy of $ C_n $ and set $c=1/(1-30\alpha)$.  Also, fix two disjoint subsets of vertices $ V_0 $ and $ V_1 $ with $ |V_0|=|V_1|=N\geq c(n+28) $ and assume that Condition  (i) does not hold. We will prove that Condition  (ii)  holds. Let $ H$ be the bipartite subgraph of $G$ where $V(H)= V_0\cup V_1$ and $E(H)=E_G(V_0, V_1) $. Since Condition~(i)  does not hold, $ H $ is an $ \alpha $-joined bipartite graph. Using Lemma \ref{expander}, there is an induced subgraph $H'=(V_0',V_1',E')$ of $H$ with $ |V_0'|, |V_1'| \geq (1-\alpha)N$ such that $ H' $ satisfies the following expansion property: For every subset $ S\subseteq V_i' $, $i\in\{0,1\}$, if $ 0<|S|\leq \alpha N $, then we have $ |N_{H'}(S)|\geq ({1-2\alpha})/({2\alpha})|S| $ and if $  |S|>\alpha N $, then we have 
	$ |N_{H'}(S)|\geq ( 1-2\alpha)N  $.

Now, let $ T_{0} $ be a rooted binary tree with $\lceil \alpha N\rceil $ leaves and height $ \lceil \log \lceil\alpha N\rceil \rceil $ and at most $ 2 \lceil\alpha N\rceil$ vertices. Let $ P $ be a path of length $ n-2-2\lceil\log\lceil\alpha N\rceil \rceil $ with endpoints $ x $ and $ y $. Now, let 
$ T_{1} $ be the rooted tree rooted at $ x $ which is obtained by identifying the root of $T_0$ and the vertex $y$. We are going to use Theorem~\ref{forst-bipartite} to prove that $H'$ contains disjoint copies of $T_0$ and $T_1$. 

Note that, $|T_{0}|+ |T_{1}|\leq 4\lceil\alpha N\rceil+n $ and $ \Delta(T_0)=\Delta(T_1)= 3 $. Moreover, we have $ d_{root}(T_{0})=2 $ and $ d_{root}(T_{1})=1 $. 
%Also, $ T_{0} $ and $ T_{1} $ are at the heights $ \lceil \log (\lceil\alpha N\rceil) \rceil $ and $ n- 2-\lceil \log (\lceil\alpha N\rceil) \rceil$, respectively. 
Since $ 0<\alpha \leq 0.033 $ and $N\geq ({n+28})/(1-30\alpha)$,  we have
\begin{enumerate}
\item $ \dfrac{1-2\alpha}{2\alpha}\geq 8 $,
\item $ (1-2\alpha)N \geq 28\lceil\alpha N\rceil+n $.
\end{enumerate}
Also, let $x_0$ and $x_1$ be two arbitrary vertices in $V'_0$.
If we apply Theorem~\ref{forst-bipartite} by setting $ m= \lfloor\alpha N\rfloor $, $ M=4\lceil\alpha N\rceil+n $ and $ X=\{x_0, x_1\} $, then $ H' $ satisfies the expansion properties of Theorem~\ref{forst-bipartite} and thus contains disjoint copies of $ T_0 $ and $ T_1 $, where $T_i$ is rooted at $x_i$, $i\in\{0,1\}$.
Let $S_0$ and $S_1$ be the set of all leaves of embedded trees $T_0$ and $T_1$, respectively (note that, $x_1$ is also a leaf of $T_1$, however we remove $x_1$ from $S_1$). Since $n$ is odd, parity of height of $T_0$ and $T_1$ are different and so $S_0\subset V_0$ and $S_1\subset V_1$ or vice versa.  Since $ H $ is an $ \alpha $-joined bipartite graph and $|S_0|,|S_1|\geq \alpha N$, there is an edge $ e=u_0u_1\in E(H) $ where $ u_i\in S_i $, $i\in \{0,1\}$.  Now if 
$ x_0 $ and $ x_1 $ are adjacent, then we can construct a cycle $ C_n $ as follows: take the path from $ u_0 $ to $ x_0 $ in $ T_0 $, then move from $ x_0 $ to $ x_1 $ and then take the path from $ x_1 $ to $ u_1 $ in $ T_1 $ and return to $u_0$ by the edge $e$ (see Figure \ref{fig2}). Since $G$ contains no cycle $C_n$, this contradiction implies that $x_0$ and $x_1$ are not adjacent. 
Now, since $x_0$ and $x_1$ were chosen arbitrarily in $V_0'$, we have $e_G(V_0')=0$. Similarly, one can prove that $e_G(V_1')=0$. Hence, Condition (ii) holds and we are done. 
\end{proof}
%%%%%%%%%%%%%%%%%%%%%%%%%%%%%%%%%%%%%%%%%%%%%%%%%%%%%%%%%%%%%%%%%%%%%%%%%%%
Now, we are ready to prove the main result of this section which provides an upper bound for $ \hat{R}(C_n, r) $, when $ n $ is odd.
\begin{theorem}\label{Odd-SRN}
For every integer $ r\geq 2 $ and odd integer $ n $, we have $ \hat{R}(C_n, r)= O(r^22^{16 r^2})n $.
\end{theorem}
\begin{proof}
First, suppose that $n$ is bounded by a constant number $n_0$. By a result from \cite{Lucss}, we have $R(C_n,r)=O(r2^r n) $. Thus, $\hat{R}(C_n,r)\leq \binom{R(C_n,r)}{2} = O(r^22^{2r} n^2)= O(r^2 2^{2r})$. 

Now, we assume that $n$ is sufficiently large. Let $ G=G(V, E) $ be the graph obtained in Lemma~\ref{quasi-graph}, with $ N=2^{8r^2} n=c_1n $ and $ p=8 r^2/n=c_2/n $. Also, let $\varepsilon=\frac{1}{2}$, $ c_3= 6\ln 2 $ and $\sqrt{6{\ln(\frac{c_1e}{c_3})/(c_2c_3)}}<\delta <1 $. Note that such $\delta $ exists since $c_3> e$ and thus, $\sqrt{6{\ln(\frac{c_1e}{c_3})/(c_2c_3)}}<\sqrt{6{(\ln c_1)/(c_2c_3)}}=\sqrt{(6\ln 2) /c_3}=1$. Hence, for every two disjoint subsets $U,W\subset V$, 
\begin{equation}\label{eq:hole}
\text{if }|U|,|W|\geq c_3 n\text{, then } e_G(U,W)\geq (1-\delta)p|U||W|>0.
\end{equation}

Now, we prove that for every positive integer $s$ and for every induced subgraph $H$ of $G$, if $|V(H)|\geq 2^{8s^2}n  $, then $H\longrightarrow(C_n)_s $. To see this, we use induction on $s$. 
Let $s=1$ and $H$ be an induced subgraph of $G$ with $|V(H)|\geq 2^8 n$. We claim that $H$ contains a copy of $C_n$. For the contrary, suppose that $H$ contains no copy of $C_n$. Then, apply Lemma~\ref{Beta-odd} with $\alpha=0.033$ and $V_0, V_1$ be two disjoint subsets of size $\lfloor |V(H)|/2 \rfloor$. Since $|V(H)|/2\geq 2^{7}n \geq 100(n+28) $ (for sufficiently large $n$), there exist two disjoint subsets $V_0'$ and $V_1'$ such that either $|V_0'|=|V_1'|\geq \alpha 2^{7} n$ and $e_H(V_0',V_1')=0$, or $|V_0'|=|V_1'|\geq (1-\alpha) 2^{7} n$ and $e_H(V_0') =e_H(V_1')=0$. In the latter case, there exists two disjoint subsets $V_0''$ and $V_0'''$ of $V_0'$ such that $|V_0''|,|V_0'''|\geq (1-\alpha) 2^{6}$ and $e(V_0'',V_0''')=0$. Both cases contradicts \eqref{eq:hole}, because $\min\{\alpha 2^{7}, (1-\alpha)2^{6}\}\geq 6\ln 2 $. This completes the proof for the case $s=1$.

Now, suppose that $ s\geq 2 $ and let  $H$ be an induced subgraph of $G$ with $|V(H)|=N'\geq 2^{8s^2} n$ and consider an arbitrary $ s $-edge coloring of $ H $ with colors $ 1, \dots, s $.  Moreover, suppose that $ H^i $, $ i\in \{1, \dots, s\} $, stands for the spanning subgraph of $ H $ induced on all $ i $-colored edges. For the contrary, suppose that for each $i\in\{1,\ldots,s\}$, $H^i$ contains no copy of $C_n$.

Let $ V_0 $ and $ V_1 $ are  two disjoint arbitrary subsets of $V(H)$ such that 
$ |V_0|=|V_1|=\lfloor N'/2\rfloor\geq 100({n+28}) $.
We do the following procedure to find two nested sequences of subsets $V_0^t\subset V_0^{t-1} \subset \cdots \subset V_0^1\subset V_0^0=V_0$ and $V_1^t\subset V_1^{t-1} \subset \cdots \subset V_1^1\subset V_1^0=V_1$, for some integer $t\leq s$ such that for each $i\in\{0,\ldots,t\}$, $ |V_0^i|=|V_1^i|\geq \alpha^i \lfloor N'/2\rfloor $.

Suppose that the subsets $ V_0^{i-1} \subset \cdots \subset V_0^0=V_0$ and $V_1^{i-1} \subset \cdots \subset V_1^0=V_1$ are chosen. Now, at step $i$, apply Lemma~\ref{Beta-odd} on the graph $H^i$ and the vertex sets $V_0^{i-1}$ and $V_1^{i-1}$. Since $|V_0^{i-1}|=|V_1^{i-1}|=\alpha^{i-1} \lfloor N'/2\rfloor\geq \alpha^s 2^{8s^2-1} n \geq 100(n+28)$, we have, either

%\textbf{Step 1}\\
%Let $ V_0 $ and $ V_1 $ are  two disjoint arbitrary subsets of vertices such that 
%$ |V_0|=|V_1|=N'\geq ({n+28})/(1-30\alpha) $, where $0< \alpha \leq 0.033 $.  Since $H^1$ contains no copy of $ C_n $, Lemma \ref{Beta-odd} implies that either,
%\begin{itemize}
%\item  [{\rm (i)}] there are two subsets $ V_{0}^1\subseteq V_0 $ and $ V_{1}^1\subseteq V_1$ with $ |V_0^1|=|V_1^1|\geq\alpha N' $ and $ e_{H^1}(V_0^1, V_1^1)=0 $, or
%\item  [{\rm (ii)}] there are two subsets $ V_0'\subseteq V_0 $ and $ V_1'\subseteq V_1$ with $ |V_0'|=|V_1'|\geq(1-\alpha) N' $ and
%$  e_{H^1}(V_0')=e_{H^1}(V_1')=0 $.
%\end{itemize}
%If (ii)  occurs, then stop. Otherwise go to step 2.
%\vspace*{0.5cm}\\
%\textbf{Step i} $ (2 \leq i \leq r) $\\
%Let $ F_i $ be the bipartite subgraph of $H$  $V(F_i)= V_0^{i-1}\cup V_1^{i-1}$ and $E(F_i)=E_H(V_0^{i-1}, V_1^{i-1}) $.
%If  there is no monochromatic copy of $ C_n $ in $ F_i\cap H^i $, then apply Lemma \ref{Beta-odd} to $ F_i\cap H^i $ with $V_0^{i-1} $ and $  V_1^{i-1} $ and $ 0<\alpha \leq 0.033 $.
%Therefore,  either
\begin{itemize}
	\item  [{\rm (i)}] there are two subsets $ V_0^i\subseteq V_0^{i-1} $ and $ V_1^i\subseteq V_1^{i-1}$ with $ |V_0^i|=|V_1^i|\geq\alpha |V_0^{i-1}|=\alpha |V_1^{i-1}| $ and $ e_{H^i}(V_0^i, V_1^i)=0 $, or
	\item  [{\rm (ii)}] there are two subsets $ V_0'\subseteq V_0^{i-1} $ and $ V_1'\subseteq V_1^{i-1}$ with $ |V_0'|=|V_1'|\geq(1-\alpha)|V_0^{i-1}|=(1-\alpha)|V_1^{i-1}| $ and
	$  e_{H^i}(V_0')=e_{H^i}(V_1')=0 $.
\end{itemize}
If (ii)  occurs, then stop the procedure. If (i) occurs, then go to step $ i+1 $ whenever $i<s$  and stop the procedure, otherwise.\\
Now, assume that the above procedure terminates in step $ t $, for some $t\leq s$. We consider the following two cases.
\vspace*{0.5 cm}\\
\textbf{Case 1.} $ t=s $.\vspace*{0.5cm}\\
\noindent In this case, we have two subsets $ V_0^s $ and $ V_1^s $  such that $|V_0^s|=|V_1^s|\geq \alpha^s \lfloor N'/2\rfloor$ and $ e_H(V_0^s, V_1^s)=0 $ which is in contradiction with \eqref{eq:hole}, since we have $ \alpha^s \lfloor N'/2\rfloor\geq \alpha^s 2^{8s^2-1}n \geq 6\ln 2n=c_3n $. \vspace*{0.5 cm}\\
\textbf{Case 2.} $ t\leq s-1 $.\vspace*{0.5cm}\\
\noindent In this case, (ii) occurs at step $t$ and so there are two subsets $ V_0'\subseteq V_0^{t-1} $ and $ V_1'\subseteq V_1^{t-1}$ with $ |V_0'|=|V_1'|\geq(1-\alpha) |V_0^{t-1}|\geq (1-\alpha) \alpha^{s-2}\lfloor N'/2\rfloor $ where $  e_{H^t}(V_0')=e_{H^t}(V_1')=0 $. Therefore, the graph $H'=H[V_0']$ contains no edge of color $t$ and its edges are colored by $s-1$ colors. On the other hand, $|V(H')|\geq (1-\alpha)\alpha^{s-2}2^{8s^2-1}n \geq 2^{8(s-1)^2}n$. Hence, by the induction hypothesis, $H'\longrightarrow (C_n)_{s-1}$. So, $H$ contains a monochromatic copy of $C_n$, a contradiction. This completes the proof of the claim.

Finally, applying the claim for $ H=G $, we have $G\longrightarrow (C_n)_r$ and so, $\hat{R}(C_n,r)\leq |E(G)|=O(c_1^2c_2) n= O(r^2 2^{16r^2})n $.
%
%at is $ H[V_0'] $ and $ H[V_1'] $ are two subgraphs of order at least $ (1-\alpha)N' $ and their edges are colored with $ s-1 $ colors $ 2, \dots, s $. By induction on $ s $, we show that $ H[V_0'] $  and similarly $ H[V_1'] $ must contain a monochromatic copy of $ C_n $.
%\vspace*{0.5 cm}\\
%\textbf{Case 2.} $ 2\leq j \leq s-1 $.\vspace*{0.5cm}\\
%\noindent In this case we have ....
%\vspace*{0.5 cm}\\
%\textbf{Case 3.} $ j=r $.\vspace*{0.5cm}\\
%\noindent In this case we have two subsets $ V_0^r $ and $ V_1^r $ with at least $ \alpha^r N'$ vertices such that $ e_H(V_0^r, V_1^r)=0 $. Which is a contradiction by Lemma \ref{quasi-graph}, since we have $ \alpha^r c=c_3 $. 
\end{proof}
%%%%%%%%%%%%%%%%%%%%%%%%%%%%%%%%%%%%%%%%%%%%%%%%%%%%%%%%%%%%%%%%%%%%%%%%%%%%%%%%%%%%%%%%%%%%%%%%%%%%%%%%%%%%%%%%%%%

\appendix
\section{Proof of Theorem~\ref{forst-bipartite}} \label{app}
Here we give a proof of Theorem~\ref{forst-bipartite}. The proof is similar to the proof of Lemma~3.1 in \cite{Ramsey good}.

%\begin{proof}
	The proof is by induction on $ \sum_{i=1}^{t}e(T_{x_i}) $. The initial case is when each tree is just a single vertex which holds by embedding $ T_{x_i} $ to $ x_i $. Now suppose that the lemma holds for all families of trees with $ \sum_{i=1}^{t}e(T_{x_i})<e $ and we have a
	family with $ \sum_{i=1}^{t}e(T_{x_i})=e>0 $. Without loss of generality, we may assume that $ e(T_{x_1} ) \geq 1 $. Let $ r $
	be the root of $ T_{x_1} $ and let $ c $ be one of its children. For every $ v\in  N_G(x_1) $ we define a set $ X^v =X\cup \{v\} $
	and a corresponding family of rooted trees $ \{T_{x}^v ~:~x\in X^v\} $ as follows. Let $ T_{x_1}^v $ be the tree rooted at $r$ obtained from  $ T_{x_1} $ by deleting $c$ and all of its descendants. Also, let $ T_{v}^v $ be the subtree of $ T_{x_1} $ rooted at $ c $ formed by $ c $ and all of its descendants. For all 
	$ x\in X^v\setminus\{x_1,v \} $, let $ T_{x}^v:=T_{x} $.\\
	We prove the following claim.
	\begin{claim}\label{CLM1}
		There is a vertex $ v\in N(x_1)\setminus X $ such that the set $ X^v $ together with the family of trees $ \{T_{x}^v ~:~x\in X^v\} $ satisfy the following for every $ S\subseteq V_i $, $ i\in \{0, 1\} $, with $ |S|\leq m $:
		\[
		|N(S)\setminus X^v|\geq 2\Delta |S\setminus X^v|+\sum_{x\in S\cap X^v}\big(d_{root}(T_x^v)+\Delta\big).
		\]
	\end{claim}
	Note that, if Claim \ref{CLM1} is true, then by induction we have an embedding of $ T_{x_1}^v, \dots, T_{x_t}^v, T_{v}^v $ into $ G $ and by adding the edge $ x_1v $, we can join the trees $ T_{x_1}^v $ and $ T_{v}^v $ in order to obtain a copy of $ T_{x_1} $ rooted at $ x_1 $. This embedding completes the proof. So it suffices to prove Claim \ref{CLM1}.\\
	
	\textit{Proof of Claim {\rm \ref{CLM1}}}. 
	By the contrary, assume that for every $ v\in N(x_1)\setminus X $ there is a set $ C_v $ with $ |C_v|\leq m $ and 
	\begin{align}\label{A.2}
	|N(C_v)\setminus X^v|\leq 2\Delta |C_v\setminus X^v|+\sum_{x\in C_v\cap X^v}\big(d_{root}(T_x^v)+\Delta\big)-1.	
	\end{align}
	Notice that taking $ S=\{x_1\} $, \eqref{A.1} implies that $ x_1 $ has at least one neighbor outside $ X $. 
	 
	Define a set of vertices $ S $ to be \textit{critical} if $|S|\leq m$ and equality holds in \eqref{A.1}.
	%\[
	%|N(S)\setminus X|= 2\Delta |S\setminus X|+ \sum_{x\in S\cap X}(d_{root}(T_x)+\Delta).
	%\]  
	We prove the following claims.
	\begin{claim}\label{CLM2}
		For every $ v\in N(x_1)\setminus X $, the set $ C_v $ is critical. Also, $ v\in N(C_v) $ and $ x_1\notin C_v $.
	\end{claim}
	\begin{proof}
		Using inequality \eqref{A.1}, we have
		\begin{align}
		|N(C_v)\setminus X^v|&\geq |N(C_v)\setminus X|-1 \label{A.3}\\
		&\geq 2\Delta |C_v\setminus X|+\sum_{x\in C_v\cap X}\big(d_{root}(T_x)+\Delta\big)-1\label{A.4}\\
		&\geq 2\Delta |C_v\setminus X^v|+\sum_{x\in C_v\cap X^v}\big(d_{root}(T^v_x)+\Delta\big)-1 \label{A.5}
		\end{align}
		Adding inequality \eqref{A.2}, implies that equality holds in each of \eqref{A.3}, (\ref{A.4}) and (\ref{A.5}). Clearly, equality in \eqref{A.4} implies that the set $ C_v $ is critical. Moreover, for equality in \eqref{A.3} implies that $ v\in N(C_v) $. For equality in (\ref{A.5}) to be hold, we must have $ x_1\notin C_v $ (since we have $ d_{root}(T_{x_1}^{v})=d_{root}(T_{x_1})-1 $).
	\end{proof}
	We will also need the following claim.
	\begin{claim}\label{CLM3} 
		For two critical sets $ S $ and $ T $, the union $ S\cup T $ is critical.
	\end{claim}
	\begin{proof}
		First notice that since $ S $ and $ T $ are critical, we have
		\begin{align}
		|N(S)\setminus X|=2\Delta | S\setminus X|+ \sum_{x\in S\cap X}\big(d_{root}(T_x)+\Delta\big),\label{A.6}\\
		|N(T)\setminus X|=2\Delta | T\setminus X|+ \sum_{x\in T\cap X}\big(d_{root}(T_x)+\Delta\big).\label{A.7}
		\end{align}
		Moreover, by applying Inequality \eqref{A.1} to $ S\cap T $ (which is smaller than $ m $ since $ S $ is critical) we have
		\begin{align}
		|N(S\cap T)\setminus X|\geq2\Delta | S\cap T\setminus X|+ \sum_{x\in S\cap T\cap X}\big(d_{root}(T_x)+\Delta\big).\label{A.8}
		\end{align}
		Also, note that by inclusion-exclusion principle, we have
		\begin{align}
		|(S\cup T)\setminus X|=|S\setminus X|+|T\setminus X|-|(S\cap T)\setminus X|, \label{A.9}
		\end{align}
		and
		\begin{align}
		\sum_{x\in (S\cup T)\cap X}\big(d_{root}(T_x)+\Delta\big)= \sum_{x\in S\cap X}\big(d_{root}(T_x)+\Delta\big)+\sum_{x\in T\cap X}\big(d_{root}(T_x)+\Delta\big)\nonumber\\
		-\sum_{x\in S\cap T\cap X}\big(d_{root}(T_x)+\Delta\big). \label{A.10}
		\end{align}
		Moreover,  we observe that 
		\begin{align*}
		|N(S\cup T)\setminus X|=|\big(N(S)\cup N(T)\big)\setminus X|,&\\
		|N(S\cap T)\setminus X|\leq|\big(N(S)\cap N(T)\big)\setminus X|,&
		\end{align*}
		which together with inclusion-exclusion principle implies
		\begin{align}
		|N(S\cup T)\setminus X|\leq |N(S)\setminus X|+|N(T)\setminus X|-|N(S\cap T)\setminus X|.\label{A.11}
		\end{align}
		Plugging (\ref{A.6}), (\ref{A.7}) and (\ref{A.8}) into (\ref{A.11}) and then using (\ref{A.9}) and (\ref{A.10}) gives
		\begin{align}
		|N(S\cup T)\setminus X|\leq 2\Delta|(S\cup T)\setminus X|+\sum_{x\in (S\cup T)\cap X}\big(d_{root}(T_x)+\Delta\big).\label{A.12}
		\end{align}
		Since both $ S $ and $ T $ are critical, we have $ |S\cup T|\leq 2m $, which together with (\ref{A.12})
		implies that 
		\[
		|N(S\cup T)|\leq |X|+|N(S\cup T)\setminus X|\leq |X|+4\Delta m+4\Delta m<M+8\Delta m.
		\]
		Therefore, by Condition~2 in the statement of Theorem~\ref{forst-bipartite}, we have $ |S\cup T|< m $. So, (\ref{A.1}) holds for the set $ S\cup T $, which together with (\ref{A.12}) implies that $ S\cup T $ is critical.
	\end{proof}
	We now resume the proof of Claim \ref{CLM1}. Let $ C=\cup_{v\in N(x_1)\setminus X} C_v $. By the above two claims, $ C $ is critical. On the other hand,  Claim \ref{CLM2} implies that $ N(x_1)\setminus X \subseteq N(C)\setminus X$ and $ x_1 \notin C$, so we have 
	\begin{align*}
	|N(C\cup \{x_1\})\setminus X|&=|N(C)\setminus X|\\
	&=2\Delta |C\setminus X|+\sum_{x\in C\cap X}\big(d_{root}(T_x)+\Delta\big)\\
	&<2\Delta |C\setminus X|+\sum_{x\in C\cap X}\big(d_{root}(T_x)+\Delta\big)+d_{root}(T_{x_1})+\Delta\\
	&=2\Delta |(C\cup \{x_1\})\setminus X| + \sum_{x\in (C\cup \{x_1\})\cap X}\big(d_{root}(T_x)+\Delta\big).
	\end{align*}
	By Inequality \eqref{A.1} we have $ |C\cup \{x_1\}|>m $, which combined with $ C $ being critical means that $ |C\cup \{x_1\}|=m+1 $. But then we have 
	\begin{align*}
	|N(C \cup \{x_1\})|\leq |X|+|N(C \cup \{x_1\})\setminus X|&\leq |X|+2\Delta m + 2\Delta (m+1)\\
	&\leq M+6\Delta m
	\end{align*}
	which contradicts the assumption of the theorem that $ |N(C \cup \{x_1\})|\geq M+8\Delta m $. Hence, the proof of Claim \ref{CLM1} and so the proof of Theorem \ref{forst-bipartite} is complete.  
%\end{proof}
\end{document}